\RequirePackage{fix-cm}
\documentclass[smallextended]{svjour3}       
\smartqed  
\usepackage{latexsym,amssymb,amsmath}
\usepackage{amssymb}
\usepackage{amsmath}
\usepackage{graphicx}
\begin{document}

\title{Slant submanifolds of Golden Riemannian manifolds
}
%



\author{Oguzhan Bahad{\i}r  \and  Siraj Uddin    
}


\institute{O.  Bahad{\i}r  \at
               Department of Mathematics,\\
 Faculty of Arts and Sciences,\\
  K.S.U. Kahramanmaras, Turkey\\
 \email{oguzbaha@gmail.com} 
          \and          
S. Uddin \at
              Department of Mathematics,\\
               Faculty of Science,\\
                King Abdulaziz University, 21589 Jeddah, Saudi Arabia\\
              \email{siraj.ch@gmail.com}      
}

\date{Received: date / Accepted: date}
\sloppy
\maketitle

\begin{abstract}
In this paper, we study slant submanifolds of Riemannian manifolds with Golden structure. A Riemannian manifold $(\tilde{M},\tilde{g},{\varphi})$ is called a Golden Riemannian manifold if the $(1,1)$ tensor field ${\varphi}$ on $\tilde{M}$ is a golden structure, that is ${\varphi}^{2}={\varphi}+I$ and the metric $\tilde{g}$ is ${\varphi}-$ compatible. First, we get some new results for submanifolds of a Riemannian manifold with Golden structure. Later we characterize slant submanifolds of a Riemannian manifold with Golden structure and provide some non-trivial examples of slant submanifolds of Golden Riemannian manifolds.
\end{abstract}

\section{Introduction}
\label{intro}

The golden ratio has fascinated Western intellectuals of diverse interests for at least 2,400 years.
Some of the greatest mathematical minds of all ages, from Pythagoras and Euclid in ancient Greece, through the medieval Italian mathematician Leonardo of Pisa and the Renaissance astronomer Johannes Kepler, to present-day scientific figures such as Oxford physicist Roger Penrose, have spent endless hours over this simple ratio and its properties. On the other hand, the fascination with the Golden Ratio is not confined just to mathematicians only but also biologists, artists, musicians, historians, architects, psychologists, and even mystics have pondered and debated the basis of its ubiquity and appeal. In fact, it is probably fair to say that the Golden Ratio has inspired thinkers of all disciplines like no other number in the history of mathematics (see \cite{livio},\cite{viki}).

In \cite{inv} C. Hretcanu and M. Crasmareanu studied the some properties of the induced structure on an invarian submanifold in a golden Riemannian manifold. \cite{golden}, M. Crasmareanu and C. Hretcanu investigated geometry of the golden structure on a manifold by using a corresponding almost product structure. In \cite{App}, C. Hretcanu and M. Crasmareanu show that a Golden Structure induces on every invariant submanifold a Golden Structure, too. In \cite{aydin}, A. Gezer, N. Cengiz, A. Salimov diccussed the problem of the integrability for Golden Riemannian structures. In \cite{m�zkan}, M. Ozkan investigated golden semi-Riemannian manifold and defines the horizontal lift of golden structure in tangent bundle.

In the end of twentieth century, B.-Y. Chen introduced the notion of slant submanifolds of almost Hermitian manifolds \cite{C1,C2}. Later, A. Lotta has extended his idea for contact metric manifolds \cite{L1} and the similar extension of slant submanifolds of $K$-contact and Sasakian manifolds has been given by Cabrerizo et al. \cite{Ca2}.

In this paper, we study  slant submanifolds of Golden Riemannian manifolds. In Section 2, we give some basic concepts. In Section 3, we get some results for submanifolds of a Riemannian manifold with Golden structure. In Section 4, we characterize slant submanifolds of a Riemannian manifold with Golden structure.  At the end of the this paper, we provide some non-trivial examples of slant Submanifolds of Golden Riemannian manifolds.

\section {Golden Riemannian manifolds}
In this section we give the some definitions and notations for Golden Riemannian manifolds.
\begin{definition}\rm{(\cite{poly},\cite{golden})
Let $(\tilde{M},\tilde{g})$ be a $(m+n)-$ dimensional
Riemannian manifold and let $F$ be a $(1,1)-$ tensor field on $\tilde{M}$. If $F$ satisfies the following equation
\begin{align}Q(X)=X^{n}+a_{n}X^{n-1}+...+a_{2}X+a_{1}I=0,\notag
\end{align}
where $I$ is the identity transformation and (for $X = F$) $F^{n-1}(p),F^{n-2}(p),...,F(p)$, $I$ are linearly independent at every
point $p\in\tilde{M}$. Then the polynomial $Q(X)$ is called the structure polynomial.}
\end{definition}
If we select the structure polynomial $Q(X) = X^{2}+I$ (or $Q(X)=X^{2}-I$ ) we get an almost complex structure (respectively, an almost
product structure).
\begin{definition}\rm{(\cite{poly},\cite{subgolden})
Let $(\tilde{M},\tilde{g})$ be a $(m+n)-$ dimensional
Riemannian manifold and let ${\varphi}$ be a $(1,1)-$ tensor field on $\tilde{M}$. If ${\varphi}$ satisfies the following equation
\begin{align}
{\varphi}^{2}-{\varphi}-I=0,\label{alt�n}
\end{align}
where $I$ is the identity transformation. Then the tensor field ${\varphi}$ is called a golden structure on $\tilde{M}$. If the Riemannian metric $\tilde g$ is ${\varphi}$ compatible, then $(\tilde{M},\tilde{g},{\varphi})$ is called a Golden Riemannian manifold \cite{golden}.}
\end{definition}
For ${\varphi}-$ compatible metric, we have
\begin{align}
\tilde{g}({\varphi}X,Y)=\tilde{g}(X,{\varphi}Y)\label{compe}
\end{align}
for any $X,Y\in\Gamma(T\tilde{M})$, where $\Gamma(T\tilde{M})$ is the set of all vector fields on $\tilde M$. If we interchange $X$ by $\varphi X$ in (\ref{compe}), then (\ref{compe}) may also be written as
\begin{align}
\tilde{g}({\varphi}X,{\varphi}Y)=\tilde{g}({\varphi}^{2}X,Y)=\tilde{g}({\varphi}X,Y)+\tilde{g}(X,Y) \label{golmet}
\end{align}
Let $\tilde M$ be an $n$-dimensional differentiable manifold with a
tensor field $F$ of type $(1,1)$ on $\tilde M$ such that $F^{2}=I$.
Then $F$ is called an almost product structure. If an almost product structure $F$
admits a Riemannian metric $\tilde g$ such that
\begin{align}
\tilde g(FX,Y)=\tilde g(X,FY),\,\forall\, X,Y\in \Gamma (T\tilde M),\notag
\end{align}
then $(\tilde M,\tilde g)$ is called almost
product Riemannian manifold.

An almost product structure $F$ induces a Golden structure as follows
\begin{align}
{\varphi}=\frac{1}{2}(I+\sqrt{5}F)\label{co}
\end{align}
Conversely, if ${\varphi}$ is a golden structure then
\begin{align}
F=\frac{1}{\sqrt{5}}(2{\varphi}-I)\label{com}
\end{align}
is an almost product structure (\cite{golden}).
\begin{example} \cite{App} \label{ex}
Consider the Euclidean $4-$space $\mathbb{R}^{4}$ with standart coordinates $(x_{1},x_{2},x_{3},x_{4})$. Let ${\varphi}$ be an $(1,1)$ tensor field on $\mathbb{R}^4$ defined by
\begin{align}{\varphi}(x_{1},x_{2},x_{3},x_{4})=(\psi x_{1},\psi x_{2},(1-\psi)x_{3},(1-\psi)x_{4})\notag
\end{align}
for any vector field $(x_{1},x_{2},x_{3},x_{4})\in \mathbb{R}^4$, where $\psi=\frac{1+\sqrt{5}}{2}$ and $1-\psi=\frac{1-\sqrt{5}}{2}$ are the roots of the equation $x^{2}=x+1$.

Then we obtain
\begin{align*}
{\varphi}^{2}(x_{1},x_{2},x_{3},x_{4})&=(\psi^{2} x_{1},\psi^{2} x_{2},(1-\psi)^{2}x_{3},(1-\psi)^{2}x_{4})\\
&=(\psi x_{1},\psi x_{2},(1-\psi)x_{3},(1-\psi)x_{4})+(x_{1},x_{2},x_{3},x_{4}).
\end{align*}
Thus, we have ${\varphi}^{2}-{\varphi}-I=0$. Moreover, we get
\begin{align*}\langle{\varphi}(x_{1},x_{2},x_{3},x_{4}),(y_{1},y_{2},y_{3},y_{4})\rangle=\langle(x_{1},x_{2},x_{3},x_{4}),{\varphi}(y_{1},y_{2},y_{3},y_{4})\rangle
\end{align*}
for each vector fields $(x_{1},x_{2},x_{3},x_{4})$, $(y_{1},y_{2},y_{3},y_{4})$$\in \mathbb{R}^4$, where$\langle\,, \rangle$ is the standard metric on $\mathbb{R}^4$. Hence, ($\mathbb{R}^4,\langle\,,\rangle,{\varphi})$ is a Golden Riemannian manifold .
\end{example}
\begin{theorem}\cite{aydin}
Let $(\tilde M,\tilde g,{\varphi})$ be a Golden Riemannian manifold. Then Golden structure ${\varphi}$ is integrable if and only if $\tilde{\nabla}\varphi=0$, where $\tilde{\nabla}$ is the Levi-Civita
connection of $\tilde g$ on $\tilde M$.
\end{theorem}

\section{Submanifolds of a Golden Riemannian manifold}

Let $(M,g)$ be a submanifold of a Golden Riemannian
manifold $(\tilde{M},\tilde{g},{\varphi})$, where $g$ is the induced metric on $M$. Then, for any $X\in \Gamma (TM)$ we can
write
\begin{align}
{\varphi}X=PX+QX,  \label{ssi-1}
\end{align}%
where $P$ and $Q$ are the projections on of $T\tilde{M}$  onto $%
TM$ and $trTM$, respectively, that is,  $PX$ and $QX$ are tangent and
transversal components  of ${\varphi}X$. For any $V\in \Gamma (TM^{\bot})$ we can
write
\begin{align}
{\varphi}V=tV+sV,  \label{ssi-2}
\end{align}
where $tV$ and $sV$ are tangent and
transversal components  of ${\varphi}V$. Then we have
\begin{align}
P^{2}=P+I-tQ, \; Q=QP+sQ, \label{ssi-3}
\end{align}
\begin{align}
s^{2}=s+I-Qt, \; t=Pt+ts. \label{ssi-4}
\end{align}
From (\ref{compe}) and (\ref{golmet}), we easily see that
\begin{align}
g(PX,Y)=g(X,PY), \label{ssi-5}
\end{align}
\begin{align}
g(PX,PY)+g(QX,QY)=g(X,Y)+g(PX,Y). \label{ssi-6}
\end{align}
If $M$ is ${\varphi}-$ invariant, then $Q=0$. Hence, from (\ref{ssi-3}) and (\ref{ssi-4}) we have
\begin{align}
P^{2}=P+I,\;  s^{2}=s+I. \label{p}
\end{align}
Therefore $(P, g)$ is golden structure on $M$. Conversely, if $(P,g)$ is a golden structure on $M$, then $Q=0$ and $M$ is ${\varphi}-$ invariant in $\tilde{M}$. In this case we have the following theorem.
\begin{theorem}
Let $(M,g)$ be a submanifold of a Golden Riemannian
manifold $(\tilde{M},\tilde{g},{\varphi})$. Then $M$ is ${\varphi}-$ invariant if and only if the induced structure $(P, g)$ of $M$ is a golden structure.
\end{theorem}

From now, we use the same symbol $g$ for the induced metric $g$ and the metric $\tilde g$. Now, let the Golden structure be integrable, that is, �$\tilde{\nabla}_{X} {\varphi}=0$, for any $X$ on $\tilde M$ where $\tilde{\nabla}$ is the Levi-Civita-connection of ${g}$. Then, the  Gauss and Weingarten formulas are respectively given by
\begin{align}
\tilde{\nabla }_{X}Y=\nabla _{X}Y+h(X,Y),\;\forall \,X,Y\in \Gamma
(TM),\   \label{6}
\end{align}
\begin{align}
\tilde{\nabla }_{X}V=-A_{V}X+\nabla _{X}^{t}V,\;\forall \,V\in \Gamma
(trTM),\   \label{7}
\end{align}
for any $X,Y\in \Gamma (TM),$ where $\nabla _{X}Y,A_{V}X$ belong to $%
\Gamma (TM)$, while $h(X,Y),\nabla _{X}^{t}V$ belong to $\Gamma (TM^{\bot})$. From the Gauss formula, we obtain
\begin{align}
\nabla_{X}{\varphi}Y+h(X,{\varphi}Y)=P\nabla_{X}Y+Q\nabla_{X}Y+th(X,Y)+sh(X,Y). \label{ssi-7}
\end{align}
Equating the tangential and normal components of Eqn. (\ref{ssi-7}), we derive
\begin{align}
\nabla_{X}{\varphi}Y=P\nabla_{X}Y++th(X,Y), \label{s6}\\
h(X,{\varphi}Y)=Q\nabla_{X}Y+sh(X,Y). \label{s66}
\end{align}
If $M$ is ${\varphi}-$ invariant then from (\ref{s6}) and (\ref{s66}), we obtain
\begin{align}
(\nabla_{X}P)Y=0,\; h(X,PY)=sh(X,Y)\label{s7}
\end{align}
From (\ref{p}) and (\ref{s7}), we have the following Proposition.
\begin{proposition}
Let $(M, g)$ be a ${\varphi}-$ invariant submanifold of a Golden Riemannian
manifold $(\tilde{M}, {g},{\varphi})$. Then the induced structure $P$ is integrable.
\end{proposition}
If $M$ is anti-invariant and ${\varphi}$ is integrable, then we get
\begin{align}
\nabla_{X}{\varphi}Y=-A_{{\varphi}Y}X+\nabla_{X}^{\bot}{\varphi}Y=Q\nabla_{X}Y+th(X,Y), \label{s6*}
\end{align}
Comparing the tangential and normal parts of (\ref{s6*}), we obtain $A_{{\varphi}Y}X=0$. Then we have the following result.
\begin{proposition}
Let  $(M,g)$ be a ${\varphi}-$ anti-invariant submanifold of a Golden Riemannian
manifold $(\tilde{M},{g},{\varphi})$. If ${\varphi}$ is integrable then $A_{{\varphi}Y}X=0$, for any $X, Y\in\Gamma(TM)$.
\end{proposition}

Now,  we compute the relations for curvature tensors with respect to the Golden structure. We know that
\begin{align*}\tilde R(X,Y)Z=\tilde\nabla_{X}\tilde\nabla_{Y}Z-\tilde\nabla_{Y}\tilde\nabla_{X}Z-\tilde\nabla_{[X,Y]}Z\end{align*}
the curvature tensor of $\tilde{M}$  with respect to Levi-civita connection $\tilde\nabla$. If ${\varphi}$ is integrable, using (\ref{compe}) and (\ref{golmet}) we obtain the following result.
\begin{proposition} \label{e�rilik}
Let  $(M,g)$ be a submanifold with curvature tensor $R$ of a Golden Riemannian manifold $(\tilde{M}, {g},{\varphi})$. If ${\varphi}$ is integrable then we have
\begin{enumerate}
\item[(i)] $R(X,Y){\varphi}={\varphi}R(X,Y)$,
\item[(ii)] $R({\varphi}X,Y)=R(X,{\varphi}Y)$,
\item[(iii)] $R({\varphi}X,{\varphi}Y)=R({\varphi}X,Y)+R(X,Y)$,
\item[(iv)] $g(R(X,Y){\varphi}Z,{\varphi}W)=g(R(X,Y)Z,{\varphi}W)+g(R(X,Y)Z,W)$,
\item[(v)] $g(R(X,Y){\varphi}Z,W)=g(R(X,Y)Z,{\varphi}W)$.
\end{enumerate}
for any $X, Y, Z, W$ tangent to $M$.
\end{proposition}
For a Riemannian manifold the Ricci tensor is defined by \cite{struc}
\begin{align}
S(X,Y)=\sum _{i=1}^{n}g(R(E_{i},X)Y,E_{i}) \label{313}
\end{align}
for any $X,Y\in \Gamma (TM)$, where $E_{1},...,E_{n}$ are local orthonormal vector fields tangent to $M$. From (\ref{compe}), (\ref{golmet}) and Proposition \ref{e�rilik} we have the following result.
\begin{proposition} \label{on}
Let  $(M,g)$ be a submanifold of a Golden Riemannian
manifold $(\tilde{M}, {g},{\varphi})$. If ${\varphi}$ is integrable then we have 
\begin{enumerate}
\item[(i)] $S({\varphi}^{2}X,Y)=S({\varphi}X,Y)+S(X,Y)$,
\item[(ii)]   $S(X,{\varphi}^{2}Y)=S(X,{\varphi}Y)+S(X,Y)$,
\item[(iii)]  $S({\varphi}X,{\varphi}Y)=S({\varphi}X,Y)+S(X,Y)$,
\item[(iv)] $S({\varphi}X,Y)=S({\varphi}Y,X)$
\end{enumerate}
for any $X, Y$ tangent to $M$.
\end{proposition}
As we know that
\begin{align}
(\nabla_{W}R)(X,Y){\varphi}Z&=\nabla_{W}(R(X,Y)Z)-R(\nabla_{W}X,Y){\varphi}Z\notag\\
&-R(X,\nabla_{W}Y){\varphi}Z-R(X,Y)\nabla_{W}{\varphi}Z\label{R}
\end{align}
and
\begin{align}
(\nabla_{Z}S)({\varphi}X,Y)=\nabla_{Z}S({\varphi}X,Y)-S(\nabla_{Z}{\varphi}X,Y)-S({\varphi}X,\nabla_{Z}Y).\label{Rs}
\end{align}
Then from Eqns. (\ref{R}), (\ref{Rs}) and Proposition \ref{e�rilik}, Proposition \ref{on}, we obtain the following proposition.
\begin{proposition}
Let  $(M,g)$ be a submanifold of a Golden Riemannian manifold $(\tilde{M}, {g},{\varphi})$. If ${\varphi}$ is integrable then we have
\begin{enumerate}
\item[(i)] $(\nabla_{W}R)(X,Y){\varphi}Z={\varphi}(\nabla_{W}R)(X,Y)Z$,\ \label{R1} 
\item[(i)] $ (\nabla_{Z}S)({\varphi}X,Y)=(\nabla_{Z}S)(X,{\varphi}Y)$,
\end{enumerate}
for any $X,Y,Z,W\in \Gamma (TM)$.
\end{proposition}
Using the Proposition \ref{e�rilik} and Proposition \ref{on} we get the following proposition.

\begin{proposition} \label{pro}
Let  $(M, g)$ be a submanifold of a Golden Riemannian manifold $(\tilde{M}, {g},{\varphi})$. If ${\varphi}$ is integrable then we have
\begin{enumerate}
\item[(i)] $(R({\varphi}X_{1},{\varphi}X_{2}).S)(X,Y)=(R({\varphi}X_{1},X_{2}).S)(X,Y)+(R(X_{1},X_{2}).S)(X,Y)$,
\item[(ii)] $(R(X_{1},X_{2}).S)({\varphi}X,{\varphi}Y)=(R(X_{1},X_{2}).S)({\varphi}X,Y)+(R(X_{1},X_{2}).S)(X,Y)$,
\end{enumerate}
for any $X_{1},X_{2},X,Y\in \Gamma (TM)$.
\end{proposition}

Let $M_{p}$ and $M_{q}$ be two real-space forms with constant sectional curvatures $c_{p}$
and $c_{q}$, respectively. Then, the Riemannian curvature tensor $R$
of a locally golden product space form ($M=M_{p}(c_{p})\times M_{q}(c_{q}),g,{\varphi})$  is given by (\cite{eyasar}):
\begin{align}
R(X,Y)Z&=\left(-\frac{(1-\psi)c_{p}-\psi c_{q}}{2\sqrt{5}}\right)\{g(Y,Z)X-g(X,Z)Y+g({\varphi}Y,Z){\varphi}X\notag\\
&-g({\varphi}X,Z){\varphi}Y\}+\left(-\frac{(1-\psi)c_{p}+\psi c_{q}}{4}\right)\{g({\varphi}Y,Z)X\notag\\
&-g({\varphi}X,Z)Y+g(Y,Z){\varphi}X-g(X,Z){\varphi}Y\}. \label{egri}
\end{align}
From (\ref{313}) and (\ref{egri}), we obtain
\begin{align}
S(Y,Z)&=\{\left(-\frac{(1-\psi)c_{p}-\psi c_{q}}{2\sqrt{5}}\right)(n-2)+\left(-\frac{(1-\psi)c_{p}+\psi c_{q}}{4}\right)trace{\varphi}\}g(Y,Z)\notag\\
&+\{\left(-\frac{(1-\psi)c_{p}-\psi c_{q}}{2\sqrt{5}}\right)(trace{\varphi}-1)\notag\\
&+\left(-\frac{(1-\psi)c_{p}+\psi c_{q}}{4}\right)(n-2)\}g({\varphi}Y,Z). \label{ric}
\end{align}
Using (\ref{ric}), we get the following result.
\begin{theorem}
Let  $M=M_{p}(c_{p})\times M_{q}(c_{q})$  be a locally Golden product space form and ${\varphi}$ is integrable. Then $M$ is Ricci symmetric.
\end{theorem}
Now, we evaluate $R.S$ for a locally Golden product space form $M=M_{p}(c_{p})\times M_{q}(c_{q})$. From (\ref{egri}) and (\ref{ric}), we derive
\begin{align}
(R(X,Y).S)(Z,W)&=-S(R(X,Y)Z,W)-S(Z,R(X,Y)W)\notag\\
&=-2\{\left(-\frac{(1-\psi)c_{p}-\psi c_{q}}{2\sqrt{5}}\right)(trace{\varphi}-1)\notag\\
&+\left(-\frac{(1-\psi)c_{p}+\psi c_{q}}{4}\right)(n-2)\}g(R(X,Y)W,{\varphi}Z).\label{30}
\end{align}
This equation gives the following theorem.
\begin{theorem}
Let  $M=M_{p}(c_{p})\times M_{q}(c_{q})$  be a locally Golden product space form and ${\varphi}$ is integrable. Then $M$ is not Ricci semi-symmetric.
\end{theorem}
Using Eqn. (\ref{30}) in Proposition \ref{pro}, we have the following consequence.
\begin{corollary}
Let  $M=M_{p}(c_{p})\times M_{q}(c_{q})$  be a locally golden product space form and ${\varphi}$ is integrable. Then
\begin{eqnarray}
(R({\varphi}X,Y).S)({\varphi}Z,W)=0.
\end{eqnarray}
\end{corollary}

\section{Slant submanifolds of a Golden Riemannian manifold}

Let  $(M,g)$ be a submanifold of a Golden Riemannian manifold $(\tilde{M},\tilde{g},{\varphi})$. For each nonzero vector $X$ tangent to $M$ at $p$, let $\theta (X)$ be the angle between $TM$ and ${\varphi}X$. If $\theta (X)$ is independent of the choice of $p\in M$ and $X\in T_{p}M$ then $M$ is called
a slant submanifold. If the slant angle $\theta=0$ and $\theta=\frac{\pi}{2}$, then $M$ is an ${\varphi}-$ invariant and ${\varphi}-$ anti-invariant submanifold, respectively. A slant submanifold which is neither invariant nor anti-invariant is called proper slant submanifold.

On the similar line of B.-Y. Chen \cite{C1,C2}, we give the following characterization of slant submanifolds in a Golden Riemannian manifold.
\begin{theorem}\label{teo}
Let  $(M, g)$ be a submanifold of a Golden Riemannian manifold $(\tilde{M}, {g},{\varphi})$. Then, $M$ is slant submanifold if and only if there exists a constant $\lambda\in [0,1]$ such that
\begin{eqnarray}
P^{2}=\lambda ({\varphi}+I), \label{slant}
\end{eqnarray}
Furthermore, if $\theta$ is slant angle of $M$, then $\lambda=cos^{2}\theta$.
\end{theorem}
\begin{proof}
Let $M$ is a slant submanifold of $\tilde{M}$. Then $\cos\theta (X)$ is independent $p\in M$ and $X\in T_{p}M$.
Therefore, from Eqns. (\ref{compe}) and (\ref{ssi-1}), we get
\begin{align}
\cos \theta(X)=\frac{g({\varphi}X,PX)}{|PX||{\varphi}X|}=\frac{g(X,{\varphi}PX)}{|PX||{\varphi}X|}. \label{cos}
\end{align}
On the other hand, by definition we have $\cos \theta(X)=|\frac{PX}{{\varphi}X}|$ and from (\ref{cos}), we derive $\cos \theta(X)=\frac{g(X,P^{2}X)}{|{\varphi}X||{\varphi}X|\cos \theta(X)}$. Thus, we obtain $\cos^{2} \theta(X)=\frac{g(X,P^{2}X)}{g(X,X)+g({\varphi}X,X)}$. Hence, we have $P^{2}=\lambda ({\varphi}+I)$.

Conversely, if we assume that $P^{2}=\lambda ({\varphi}+I)$, then we obtain $\lambda=cos^{2}\theta$, i.e., $\theta(X)$ is constant on $M$ and hence $M$ is slant, which proves the theorem completely.
\end{proof}
Using Eqn. (\ref{alt�n}) we have the following consequence of the above theorem.
\begin{corollary}
Let  $(M, g)$ be a submanifold of a Golden Riemannian manifold $(\tilde{M}, {g},{\varphi})$. Then, $M$ is a slant submanifold if and only if there exists a constant $\lambda\in [0,1]$ such that
\begin{align}
{\varphi}^{2}=\frac{1}{\lambda} P^{2},
\end{align}
where $\lambda=cos^{2}\theta$ and $\theta$ is the slant angle of $M$.
\end{corollary}
\begin{lemma}
Let  $(M,g)$ be a slant submanifold of a Golden Riemannian manifold $(\tilde{M},\tilde{g},{\varphi})$. Then, for any $X,Y\in\Gamma(TM)$, we have
\begin{eqnarray}
g(PX,PY)=cos^{2}\theta(g(X,Y)+g(X,PY)),\label{kos}\\
g(QX,QY)=sin^{2}\theta(g(X,Y)+g(PX,Y)).\label{sin}
\end{eqnarray}
\end{lemma}
\begin{proof}
From (\ref{ssi-5}) and (\ref{slant}), we obtain
$$g(PX,PY)=g(X,\lambda \varphi{Y}+\lambda Y)=cos^{2}\theta(g(X,Y)+g(X,PY)).$$
Moreover, from (\ref{ssi-6}) and (\ref{kos}), we derive
$$g(QX,QY)=g(X,Y)+g(PX,Y)-g(PX,PY)=sin^{2}\theta(g(X,Y)+g(PX,Y)).$$
Hence, the proof is complete.
\end{proof}

Now, we construct some non-trivial examples of slant submanifolds of a Riemannian manifold with Golden structure.
\begin{example}
\rm{Consider a submanifold $M$ of Euclidean 4-space $\mathbb{R}^4$ given by the following immersion
\begin{align*}x(u_{1},u_{2})=(u_{1}\cos\theta, u_{1}\sin\theta, u_{2}, 0).
\end{align*}
Then the tangent space $TM$ is spanned by the following vector fields
\begin{align*}e_{1}=(\cos\theta,\sin\theta,0,0),\,\,e_{2}=(0,0,1,0).
\end{align*}
Now, we consider the Golden structure from Example \ref{ex}. Then, we obtain
\begin{align*}{\varphi}e_{1}=(\psi \cos\theta,\psi \sin\theta,0,0),\,\,\, {\varphi}e_{2}=(0,0,1-\psi,0).
\end{align*}
 Thus, we derive
\begin{align*}\langle{\varphi}e_{1},e_{1}\rangle=\psi,\;\langle{\varphi}e_{2},e_{2}\rangle=1-\psi,\;\langle{\varphi}e_{1},e_{2}\rangle=0
\end{align*}
 and
\begin{align*}Pe_{1}=\psi e_{1},\;Pe_{2}=(1-\psi)e_{2} .
\end{align*}
If $\Theta$ is the slant angle of $M$, then we get $\cos\Theta=1$, thus $M$ is a ${\varphi}-$ invariant submanifold.}
\end{example}

\begin{example}
\rm{Consider the Euclidean $4-$space $\mathbb{R}^4$ with standart coordinates $(x_{1},x_{2},x_{3},x_{4})$. Let ${\varphi}$ be an $(1,1)$ tensor field on $\mathbb{R}^4$ given by
\begin{align*}
{\varphi}(x_{1},x_{2},x_{3},x_{4})=(\psi x_{1},(1-\psi) x_{2},\psi x_{3},(1-\psi)x_{4})
\end{align*}
for any $(x_{1},x_{2},x_{3},x_{4})\in \mathbb{R}^4$, where $\psi=\frac{1+\sqrt{5}}{2}$ and $1-\psi=\frac{1-\sqrt{5}}{2}$ are the roots of the equation $x^{2}=x+1$.
Then, we obtain
\begin{align*}
{\varphi}^{2}(x_{1},x_{2},x_{3},x_{4})&=(\psi^{2} x_{1},(1-\psi)^{2} x_{2},\psi^{2}x_{3},(1-\psi)^{2}x_{4}),\\
&=(\psi x_{1},(1-\psi) x_{2},\psi x_{3},(1-\psi)x_{4})+(x_{1},x_{2},x_{3},x_{4}).
\end{align*}
Thus, we have ${\varphi}^{2}-{\varphi}-I=0$. Moreover, the metric $\langle\;\rangle$ is ${\varphi}-$ compatible. Hence, $(\mathbb{R}^4,\langle\:\rangle,{\varphi})$ is a Golden Riemannian manifold .Now, consider a submanifold $M$ of $\mathbb{R}^4$ given by the immersion
\begin{align*}
x(u_{1}, u_{2})=(\psi u_{1},(1-\psi)u_{1},\psi u_{2},(1-\psi) u_{2}).
\end{align*}
Then we have 
\begin{align*}
e_{1}=(\psi,1-\psi,0,0),\,\,\, e_{2}=(0,0,\psi,1-\psi)
\end{align*}
and
\begin{align*}
{\varphi}e_{1}=(\psi+1,2-\psi,0,0),\,\,\, {\varphi}e_{2}=(0,0,\psi+1,2-\psi).\end{align*}
 Thus, we derive
\begin{align*}
\langle{\varphi}e_{1},e_{1}\rangle=4,\;\,\langle{\varphi}e_{2},e_{2}\rangle=4,\;\,\langle{\varphi}e_{1},e_{2}\rangle=0
\end{align*}
and
\begin{align*}Pe_{1}=\frac{4}{3} e_{1},\;Pe_{2}=\frac{4}{3}e_{2} .
\end{align*}
Then $M$ is a slant submanifold with slant angle $\Theta=\cos^{-1}\left(\frac{4}{\sqrt{21}}\right)$.}
\end{example}

\begin{example}
\rm{Consider the Euclidean $4-$space $\mathbb{R}^4$ with standart coordinates $(x_{1},x_{2},x_{3},x_{4})$. Let ${\varphi}$ be an $(1,1)$ tensor field on $\mathbb{R}^4$ defined by
$${\varphi}(x_{1},x_{2},x_{3},x_{4})=((1-\psi)x_{1}, (1-\psi)x_{2},\psi x_{3},\psi x_{4})$$
for every point $(x_{1},x_{2},x_{3},x_{4})\in \mathbb{R}^4$, where $\psi=\frac{1+\sqrt{5}}{2}$ and $1-\psi=\frac{1-\sqrt{5}}{2}$ are the roots of the equation $x^{2}=x+1$. Then it is easy to see that $\varphi$ is a Golden structure on $\mathbb{R}^4$ with ${\varphi}-$ compatible metric $\langle\;\rangle$. Hence, $(\mathbb{R}^4,\langle\:\rangle,{\varphi})$ is a Golden Riemannian manifold. 

Consider a submanifold $M$ of $\mathbb{R}^4$ given by
\begin{align*}
x(u_{1},u_{2})=(k\psi u_{1},k\psi u_{2},(1-\psi)u_{1},(1-\psi) u_{2}),
\end{align*}
for any $k\neq0, 1$ .
Then we have $e_{1}=(k\psi,0,1-\psi,0)$, $e_{2}=(0,k\psi,0,1-\psi)$, ${\varphi}e_{1}=(-k,0,-1,0)$, ${\varphi}e_{2}=(0,-k,0,-1)$. Then, we obtain
\begin{align*}
\langle{\varphi}e_{1},e_{1}\rangle=\langle{\varphi}e_{2},e_{2}\rangle=-1+\psi-k^{2}\psi,\;\langle{\varphi}e_{1},e_{2}\rangle=0.
\end{align*}
If $\theta$ is the slant angle of $M$, then $M$ is a slant submanifold with slant angle $\theta=\cos^{-1}\left(\frac{-1+\psi-k^{2}\psi}{\sqrt{k^{2}+1}}\right)$.}
\end{example}
Now, we give another useful result for slant submanifolds of Golden Riemannian manifolds.
\begin{theorem}
Let  $(M, g)$ be a submanifold of Golden Riemannian manifold $(\tilde{M},\tilde{g},{\varphi})$. Then $M$ is proper slant submanifold of $\tilde{M}$ if and only if there exists a constant $k \in [0,1]$ such that
\begin{align}
tQX=k(P+I)-(1-k)Q
\end{align}
 for any $X,Y\in \Gamma(TM)$. Furthermore $k=\sin^{2}\theta$ and $\theta$ is the slant angle of $M$.
\end{theorem}
\begin{proof}
 From (\ref{ssi-3}) we know that
 \begin{align}
tQX=-P^{2}X+PX+X \label{denk}
\end{align}
for any $X\in\Gamma(TM)$. If $M$ is a slant submanifold, then using (\ref{ssi-1}) and (\ref{slant}), we obtain
\begin{align*}
tQX &=-\lambda ({\varphi}X+X)+PX+X, \\
&=-\lambda (PX+X)-\lambda QX+PX+X,\\
&=(1-\lambda)(PX+X)-\lambda QX.
\end{align*}
Conversely, we suppose that $tQX=k(P+I)-(1-k)Q,\;k \in [0,1]$.  Then from Eqns. (\ref{ssi-1}) and (\ref{ssi-3}), we derive
\begin{align*}
P^{2}X&=PX+X-tQX \\
&=PX+X+(1-k)QX-k(PX+X),\\
&=(1-k)({\varphi}X+X).
\end{align*}
If we put $(1-k)=\lambda=cos^{2}\theta$, then $M$ is a slant submanifold. Hence, the theorem is proved completely.
\end{proof}




\end{document}